\documentclass[a4paper, 11pt,reqno]{amsart}

\usepackage{graphicx}
\graphicspath{{IMAGES/}{IMAGES1/}{./}}
\usepackage{amsmath, amsfonts, amssymb, amsthm}
\usepackage{xcolor}
\usepackage[margin=0.75in]{geometry}

\usepackage{booktabs} 

\usepackage[utf8]{inputenc}

\newcommand{\norm}[1]{\left|\hspace*{2.5pt} \!\!\left| #1\right|\hspace*{2.5pt} \!\!\right|}
\renewcommand{\pmod}[1]{\left( \mathrm{ mod\;}#1\right)}

\newcommand{\Z}{\mathbb{Z}} 

\newcommand{\R}{\mathbb R}
\newcommand{\N}{\mathbb{N}}
\newcommand{\F}{\mathbb{F}}
\newcommand{\NF}{\operatorname{NF}}
\newcommand{\cN}{\mathcal N}
\newcommand{\cP}{\mathcal P}
\newcommand{\cU}{\mathcal U}
\newcommand{\cJ}{\mathcal{J}}
\newcommand{\cS}{\mathcal S}
\newcommand{\abs}[1]{\left|#1\right|} 

\newcommand{\floor}[1]{\left\lfloor #1 \right\rfloor}

\newcommand{\cR}{\mathcal{R}}

\newcommand{\AJh}{N(q,h)}
\newcommand{\cA}{\mathcal A}

\usepackage{bm}
\newcommand{\br}{\bm{r}}

\newtheorem{theorem}{Theorem}
\newtheorem{lemma}{Lemma}
\newtheorem{corollary}{Corollary}

\newtheorem{remark}{Remark}

\theoremstyle{remark}

\usepackage{subfigure}  
\usepackage{float}
\usepackage[]{caption}
\definecolor{orange}{rgb}{1,0.5,0}
\definecolor{Red}{rgb}{.795,0.015,0.017}
\definecolor{Ggreen}{rgb}{0.,0.675,0.0128}
\definecolor{Bblue}{rgb}{0.16,.32,0.91}

\usepackage[hyphens]{url}
\usepackage{multicol}

 \usepackage[backref=page]{hyperref}
\hypersetup{
    colorlinks = true,
linkcolor={Red},
urlcolor={blue},
citecolor={Ggreen},    
urlcolor = {blue},
citebordercolor = {0.33 .58 0.33},
 linkbordercolor = {0.99 .28 0.23},
 breaklinks=true
}

\begin{document}

\title[Angular distribution towards the points of the neighbor-flips modular curve]{Angular distribution towards the points of the neighbor-flips modular curve 
seen by a fast moving observer}
\author{Jack Anderson, Florin P. Boca, Cristian Cobeli, Alexandru Zaharescu}


\address[Jack Anderson]{Department of Mathematics, University of Illinois at Urbana-Champaign, Urbana, IL 61801, USA.}
\email{jacka4@illinois.edu}

\address[Florin P. Boca]{Department of Mathematics, University of Illinois at Urbana-Champaign, Urbana, IL 61801, USA.}
\email{fboca@illinois.edu}

\address[Cristian Cobeli]{"Simion Stoilow" Institute of Mathematics of the Romanian Academy,~21 Calea Grivitei Street, P. O. Box 1-764, Bucharest 014700, Romania}
\email{cristian.cobeli@imar.ro}

\address[Alexandru Zaharescu]{Department of Mathematics,University of Illinois at Urbana-Champaign, Urbana, IL 61801, USA,
%
and 
"Simion Stoilow" Institute of Mathematics of the Romanian Academy,~21 
Calea Grivitei 
Street, P. O. Box 1-764, Bucharest 014700, Romania}
\email{zaharesc@illinois.edu}




\begin{abstract}
Let $h$ be a fixed non-zero integer.
For every $t\in \mathbb{R}_+$ and every prime $p$, consider the angles between rays from an observer
located at the point $(-tJ_p^2,0)$ on the real axis towards the set of all integral solutions $(x,y)$ of
the equation $y^{-1}-x^{-1}\equiv h \pmod{p}$ in the square $[-J_p,J_p]^2$, where $J_p=(p-1)/2$.
This set of points can be seen as a generic model for any target set with points randomly distributed on the integer coordinates of a square, 
in which, apart from a small number of exceptions, exactly one point lies above any abscissa.

We prove the existence of the limiting gap distribution for this set of angles as $p\rightarrow \infty$,
providing explicit formulas for the corresponding density function, which turns out to be independent of $h$.
The resulted gap distribution function shows the existence of a sequence of threshold points between which the distribution of seen angles has different shapes.
This provides a tool of reference in guiding the observer, which allows one to find and control the position relative to the universe of observed points.
\end{abstract}





\maketitle

\bigskip

\noindent

\bigskip

\section{Introduction}\label{S:Introduction}

Some time ago, the last three authors~\cite{BCZ2000} investigated the angular gap distribution of integer lattice
points in dilated regions $J\Omega$, $J\to\infty$, where $\Omega\subseteq\R^2$ is a star-shaped bounded domain with
respect to $O = (0, 0)$, with piece-wise $C^1$ boundary, and the points are seen from an observer
located at the origin. The methods were number theoretical, ultimately relying on 
Weil--Estermann bounds for Kloosterman exponential sums. The particular case $\Omega=\{(x,y): 0<y<x<1\}$ is highly relevant, because the slopes of the segment lines $OP$,
$P\in J\Omega\cap\Z^2$, represent exactly the Farey points of height $J$ in the interval $[0, 1]$, which constitute a significant object in Number Theory, Ergodic Theory, and Topology.
This work has
generated significant interest, leading to deep connections with homogeneous dynamics and the
theory of the periodic Lorentz gas established by Marklof and Str\"ombergsson~\cite{MS2010}, to the
horocyclic flow on the modular surface found by Athreya and Cheung~\cite{AC2014} and developed more recently by Marklof and Welsh \cite{MW2023}
(see also Bonanno et al.~\cite{BDI2022} and references therein for the latest developments), and to new statistical properties
of hyperbolic lattice points (see~\cite{BPZ2014,CKLW2021,KK2015,MV2018,PP2022,RS2017}). 
In particular, the
situation where the observer is located at a fixed point, other than the origin inside the unit
square, and $\Omega$ is the unit disk was considered in
an article of Marklof and Str\"ombergsson~\cite{MS2010}, 
where the corresponding limiting gap
distribution was shown to coincide with the Elkies-McMullen gap distribution of the sequence $\sqrt{n}\pmod{1}$~\cite{EM2004} when the observer is located at a fixed point with at least one irrational coordinate.

Very recently, the authors investigated a new model~\cite{ABCZ2023}, where the observer is no longer fixed as $J\to\infty$. Instead, the observer moves away from the origin along the horizontal axis with constant acceleration, while $\Omega$ is the unit square $[-1,1]^2$. In this situation the existence of the limiting gap distribution as $J\to\infty$ was established and an algorithm was provided to compute its general formula.

In Granville et al.~\cite{GSZ2005}, the distribution of points on some algebraic curve modulo a large prime $p$ was studied. For a survey of results in that area, in particular for results on points on modular hyperbolas, the reader is referred to Shparlinski~\cite{Shp2012}. In particular, for distribution of inverses, see~\cite{CVZ2003, CZ2000, CZ2001}. In the current paper, we consider a related model with arithmetic flavor, with the moving observer only looking at a certain subset of the integer points inside $J\Omega$. In particular, if $q$ is a positive integer, we take 
$J=J_q=(q-1)/2$ and consider the points $([n]^{-1}, [n+h]^{-1})$ for a fixed integer $h$, 
where $n\in\Z/q\Z$, and $[n]^{-1}$ represents the inverse of $[n]$, the representative of $n\pmod{q}$, in $[-J_q, J_q]$. We write
\begin{equation*}
    \cA(q, h) = \big\{([n]^{-1}, [n+h]^{-1}) : n\in\Z/q\Z,\ \gcd(n,q)=\gcd(n+h,q)=1\big\} \subseteq [-J_q,J_q]^2,
\end{equation*}
and denote $\AJh=\#\cA(q,h)$.
When $q=p$ is a prime and $p\nmid h$, we have $N(p,h)=p-2$.

This is the \textit{neighbor-flips} modular curve, which was chosen because it provides a
suitable model for randomly distributed point sets where, 
with only a few exceptions, precisely one point is located above every abscissa.
On the other hand, the arithmetic properties of the curve in a finite field allow 
to employ the necessary estimates of exponential sums (see Section~\ref{SectionDistributionRationalFunctions})
essential in the explicit calculation 
of the limiting gap distribution function.

Let us remark that the sets $\cA(q,h)$ are the elements of a partition of the translated square $(\Z/q\Z)^2-(J_q,J_q)$
from which the vertical and the horizontal lines whose order are not relatively prime to $q$ are excluded.
Note that the union over $h$ of the sets $\cA(q,h)$ is largest when $q$ is prime,
resulting in only one row and one column remaining uncovered in the square $(\Z/q\Z)^2$. 
At the other end, when $q$ is highly composite, 
points from periodic vertical and horizontal friezes of various widths are excluded from 
$(\Z/q\Z)^2$. Figure~\ref{FigureNFal} illustrates two contrasting such examples with two consecutive $q$'s.

In addition, it is worth noting that the properties of the inverses (see~\cite{CVZ2003, CZ2000,CZ2001}) lead to the sets $\cA(p,h)$ 
with~$p$ prime 
becoming uniformly distributed in $(\Z/p\Z)^2$ 
when $p$ is sufficiently large, for all $h$, except in cases where $p$ divides~$h$. 
Then, particularly when $h=0$, the modular curve $\cA(p,0)$ covers only
the first diagonal of $(\Z/p\Z)^2$ and, as a result, 
it is not evenly distributed in the whole square.

For a fixed $t>0$, we place our observer at the point $P_{t,J} = (-tJ^2, 0)$
(see Figure~\ref{Figure-pe359he26-s} for an example of such a position where the observer is looking towards the points of curve $\cA(q,h)$).
Consider the angles $\angle (PP_{t,J}O)$ between rays from $P_{t,J}$ to points of $\cA(q,h)$ and the $x$-axis (with points below the $x$-axis being ascribed a negative angle) and label them in ascending order as
$
\alpha_{\text{min}}= \alpha_1 \leq \alpha_2 \leq \cdots \leq \alpha_{\AJh} =\alpha_{\text{max}}.
$
\begin{center}
\begin{figure}[htb]
\centering
\hfill
   \includegraphics[angle=0,width=0.77\textwidth]{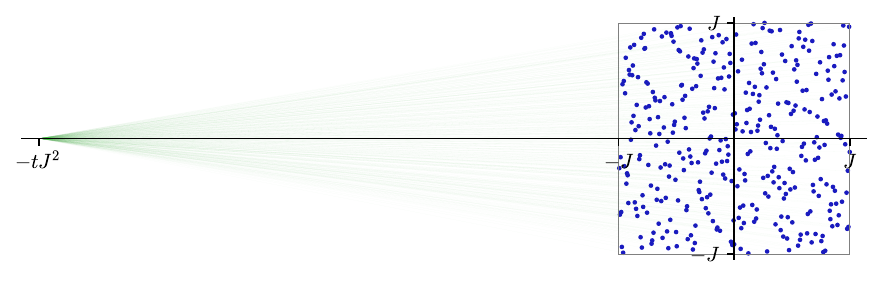}  
\hfill\mbox{}
\vspace*{-3mm}
\caption{The points of $\cA(359,26)$, seen by an observer placed at $(-tJ_{359}^2,0)$.
}
 \label{Figure-pe359he26-s}
 \end{figure}
\end{center}

Furthermore, we write the average angle between consecutive points as
\begin{equation*}
    \Delta_{av} := \frac{1}{\AJh-1}\sum_{j=1}^{\AJh-1} (\alpha_{j+1}-\alpha_j) = 
\frac{\alpha_{\operatorname{max}} -\alpha_{\operatorname{min}}}{\AJh-1}\,.
\end{equation*}
We can now define our \emph{gap-distribution function} by
\begin{equation*}
G_{t,q,h}(\lambda) := \frac{1}{\AJh-1}\#\left\{j\in\{1,\dots,\AJh-1\} : \alpha_{j+1}-\alpha_j \geq \lambda\Delta_{av}\right\}.
\end{equation*}

The limit $\lim_{q\to\infty}G_{t,q,h}(\lambda)$ may not exist in general 
(see Figure~\ref{F:GLambdaNumerical}, $G_{t,q,h}(\lambda)$ can vary wildly from one value of~$q$ to the next), 
but can exist when restricted to certain families. For any subset \mbox{$\cN\subseteq\N$} 
with infinite cardinality, we can write $G_{t,\cN,h}(\lambda) = \lim_{n\to\infty}G_{t,q_n}(\lambda)$, 
where $\cN=\{q_1, q_2, \dots\}$ when written in increasing order.
\begin{figure}[htb]
\centering
\includegraphics[width=0.48\textwidth]{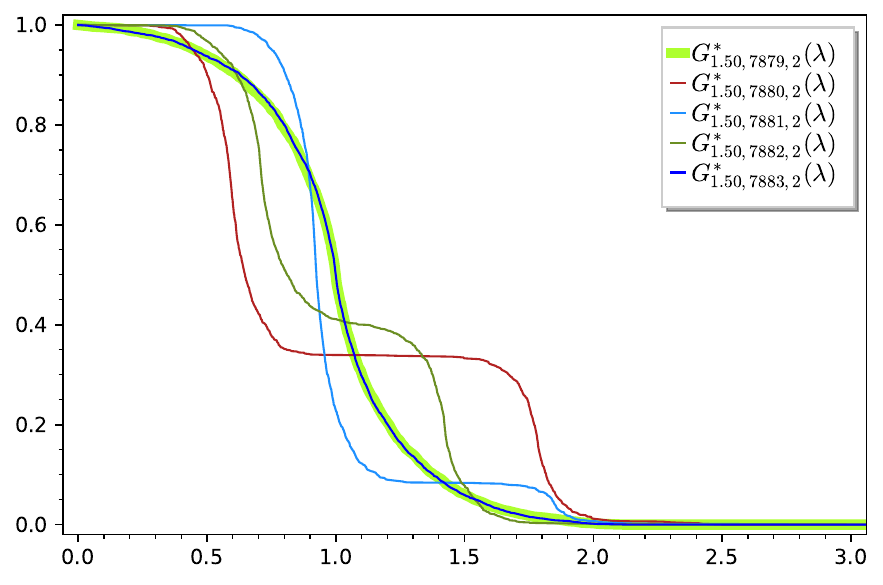}
\hfill
\includegraphics[width=0.48\textwidth]{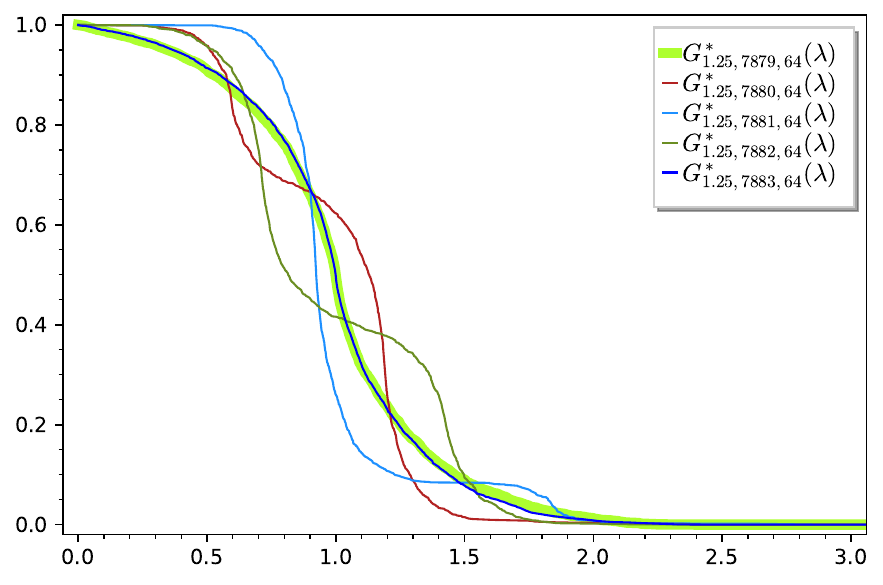}
\hfill\mbox{}
\vspace*{-3mm}
\caption{
Graphs of $G_{t,q,h}(\lambda)$ computed numerically
and thus referred to as $G^*_{t,q,h}(\lambda)$.
Five graphs are shown in each image, for $q=7879, 7880,7881,7882$ and $7883$. 
Notice that $q=7879$ and $7883$ are both prime 
and the difference between the corresponding graphs is indiscernible.
In the image on the left, $h=2$ and $t=1.5$, while in the one on the right, $h=64$ and $t=1.25$.}
\label{F:GLambdaNumerical}
\end{figure}

In the current paper, we will focus on the family $\cP$ of prime numbers and $q=p\in \cP$, $p> \lvert h\rvert >0$. In this situation $N(p,h)=p-2=2J-1$. 
Employing the
Taylor formula for $\tan x$ and $\arctan x$ we see that
$\alpha_{\operatorname{max}} =1/(tJ) +O_t (1/J^2)=
-\alpha_{\operatorname{min}}$, and so
\begin{equation*}
    \Delta_{av} = 
    \frac{\alpha_{\operatorname{max}}-\alpha_{\operatorname{min}}}{\AJh-1}=\frac{1}{tJ^2} 
+ O_t (J^{-3}).
\end{equation*}
The elements of the sequence of finite sets 
$\{ \alpha_j\}_{j=1}^{\AJh}$ are uniformly distributed as 
$p\rightarrow \infty$, in the sense that
\begin{equation*}
\lim\limits_{\substack{p\rightarrow\infty \\ p\in \cP}}
\frac{1}{p-2} \# \left\{ j: 1\leq j\leq p-2, \ 
\alpha \leq \frac{\alpha_j -\alpha_{\operatorname{min}}}{\alpha_{\operatorname{max}}-\alpha_{\operatorname{min}}} 
\leq \beta\right\} 
=\beta -\alpha,\quad \text{for } 0\leq\alpha 
\leq \beta \leq 1.
\end{equation*}
We will show that the limiting gap distribution exists for this family of angles. We will also see that in this case the limiting function does not depend on $h$, as shown by the following result, and so we will denote it simply by $G_{t,\cP}(\lambda)$.

\begin{theorem}\label{Thm1.1}
Let $h\neq 0$ and $D>0$ be fixed integers. Consider the set of points 
\begin{equation*}
\Omega (t,\lambda):=\bigg\{(x, y_{-D+1}, \dots, y_{D})\in\bigg[
-\frac{1}{2},\frac{1}{2}\bigg]^{2D+1}:
        \; x\geq 0, \ 
        y_j \not\in \left[y_0 + (j-\lambda)\frac{t}{4x}, y_0 + j\frac{t}{4x}\right] \text{ for }\ j\neq 0 \bigg\}.
    \end{equation*} 
 Then, for any $t\in (1/D, 1/(D-1)]$ and
    $\lambda >0$, as $p\in \cP$, $p\rightarrow \infty$, we have 
\begin{equation*}
  \begin{split}
     G_{t,p,h}(\lambda) &= 2\mu\big(\Omega(t,\lambda)\big) + O_{\lambda, t}\left(p^{-1/(4D+4)}
     \log^{(2D+1)/(2D+2)}p\right),
  \end{split}
\end{equation*}
and
\begin{equation*}
      G_{t,\cP}(\lambda) = 2\mu\big(\Omega(t,\lambda)\big),
\end{equation*}
where $\mu$ is the Lebesgue measure in ${\mathbb R}^{2D+1}$.
\end{theorem}

For $t\ge 1$, we computed explicitly the expression of $G_{t,\cP}(\lambda)$ 
as a sum of compositions of elementary functions. 
For $t \ge 2$, this is given in the next corollary, 
while for $1\le t < 2$ in Remark~\ref{RemarkA} from the Appendix.
It turns out that the domain of pairs $(t,\lambda)\in [1,\infty)\times [0,\infty)$
is split in polygonal curved regions (see~Figure~\ref{FigureDomains}) 
on which the expression of $G_{t,\cP}(\lambda)$  remains unchanged. 
Moreover, the expressions show an incipient tendency of flow 
and spread into certain nearby regions. 
It is likely that the phenomenon may act similarly for $t<1$.
Overall, the formulas fit together harmoniously so that $\lambda\mapsto G_{t,\cP}(\lambda)$ is $C^1$ for 
$\lambda\ge 0$, 
except for the spike at $\lambda=1$.

\begin{corollary}\label{CorollaryG1}
    When $t\ge 2$,
    \begin{equation*}
        G_{t,\cP}(\lambda) 
        = \begin{cases}
            1 & \text{if } \lambda\leq 1-\frac{2}{t}; \\[1mm]
            \frac{1}{8}\left(4 + t^2 -2\lambda t^2 + \lambda^2 t^2 + 4t(1-\lambda)\log\left(\frac{2}{t(1-\lambda)}\right)\right) & \text{if } 1-\frac{2}{t}\leq\lambda<1; \\[1mm]
            \frac{1}{8}\left(4 -t^2+2\lambda t^2 -\lambda^2t^2 - 4t(\lambda-1)\log\left(\frac{2}{t(\lambda-1)}\right)\right) & \text{if } 1\leq \lambda<1+\frac{2}{t}; \\[1mm]
            0 &\text{if }  1+\frac{2}{t}\le \lambda.
        \end{cases} 
    \end{equation*} 
    \end{corollary}

Figure~\ref{F:GtPLambda276} illustrates the graph of the limit $G_{t,\cP}(\lambda)$ obtained in Corollary~\ref{CorollaryG1} 
for various values of~$t>2$.
\begin{figure}[htb]
\centering
\hfill
\includegraphics[width=0.48\textwidth]{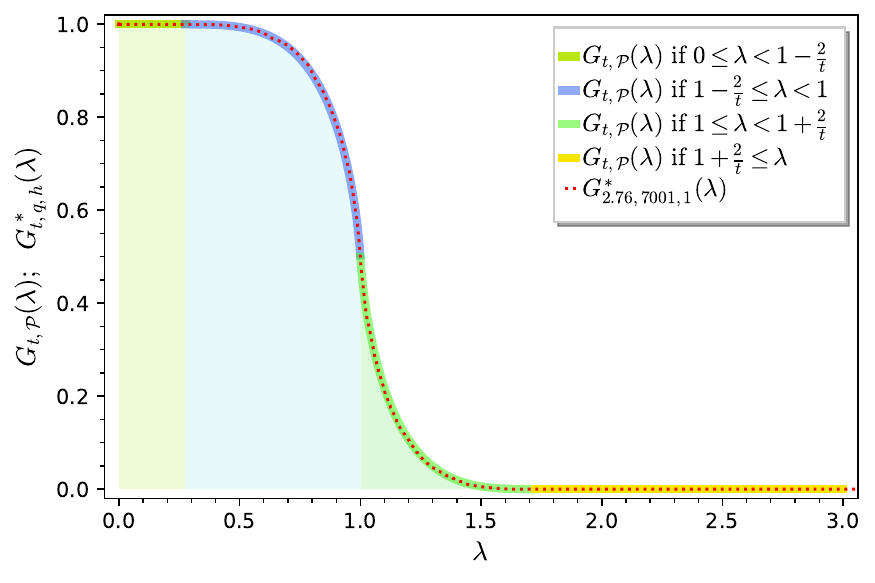}
\hfill
\includegraphics[width=0.48\textwidth]{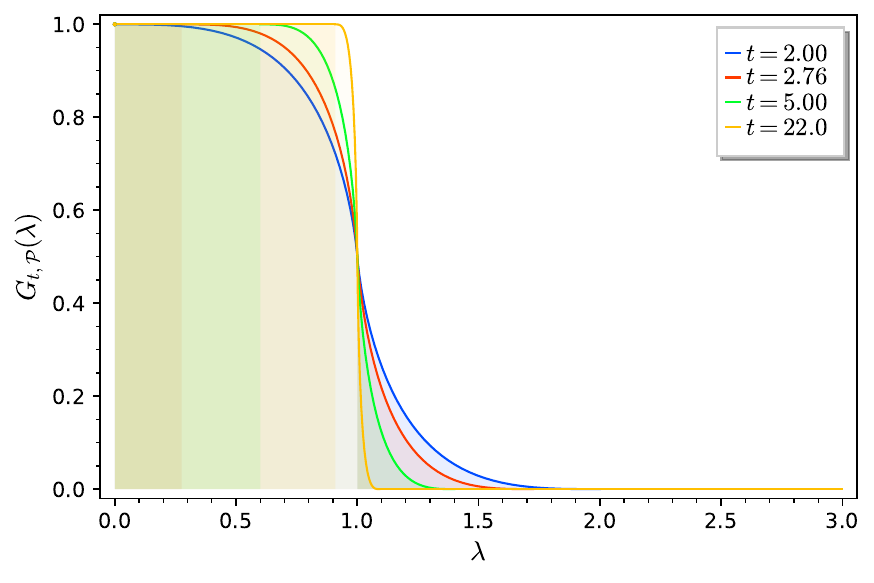}
\hfill\mbox{}
\caption{The graph of the limit $G_{t,\cP}(\lambda)$ for $t\ge 2$.
In the image on the left-hand side, the graph of $G^*_{t,q,h} (\lambda)$, calculated numerically, 
is overlaid with dots. Here $t = 2.76, p=10\,007$ and $h=1$.
In the image on the right-hand side, the graph of $G_{t,\cP}(\lambda)$ is 
shown for four different values of $t$, that is, $t=2,2.76,5$ and $22$.
The distinct expressions that $G_{t,\cP}(\lambda)$ 
takes for $\lambda \ge 0$ are indicated 
by different shades of the corresponding subgraph regions.}
\label{F:GtPLambda276}
\end{figure}

We can also define a density function
$ g_{t,\cP,h}(\lambda)$ as follows.
Consider the set of the gaps of size about $\lambda\ge 0$ times the average, which is given by
\begin{equation*}
    \cU(\delta,t,\cP,h):=
    \Big\{j\in\{1, \dots, \AJh-1\} : 
    \frac{\alpha_{j+1}-\alpha_j}{\Delta_{av}} 
    \in [\lambda-\delta, \lambda+\delta)\Big\}\,.
\end{equation*}
(Recall that the angles $\alpha_j$ depend on the position of the observer at $P_{t,J_p}=(-tJ_p^2,0)$, so that $\cU(\delta,t,\cP,h)$
implicitly depends on $t$.)
Then the density is
\begin{equation*}
    g_{t,\cP,h}(\lambda)
    :=\lim_{\delta\to0}\frac{1}{2\delta}
    \lim_{\substack{p\to\infty \\ p\in \cP}} 
    \frac{1}{\AJh-1}
    \#\cU(\delta,t,\cP,h).
\end{equation*}
 
\begin{figure}[hb]
\centering
\hfill
\includegraphics[width=0.48\textwidth]{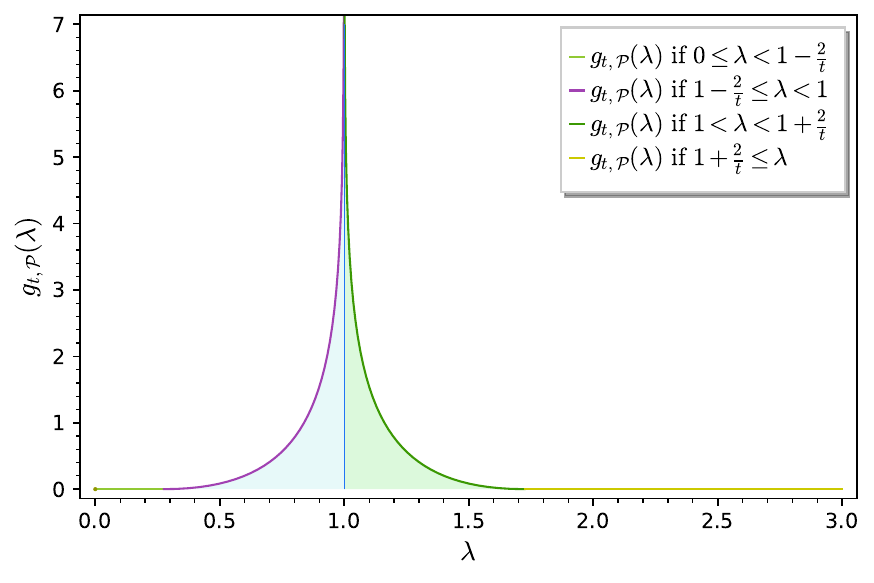}
\hfill
\includegraphics[width=0.48\textwidth]{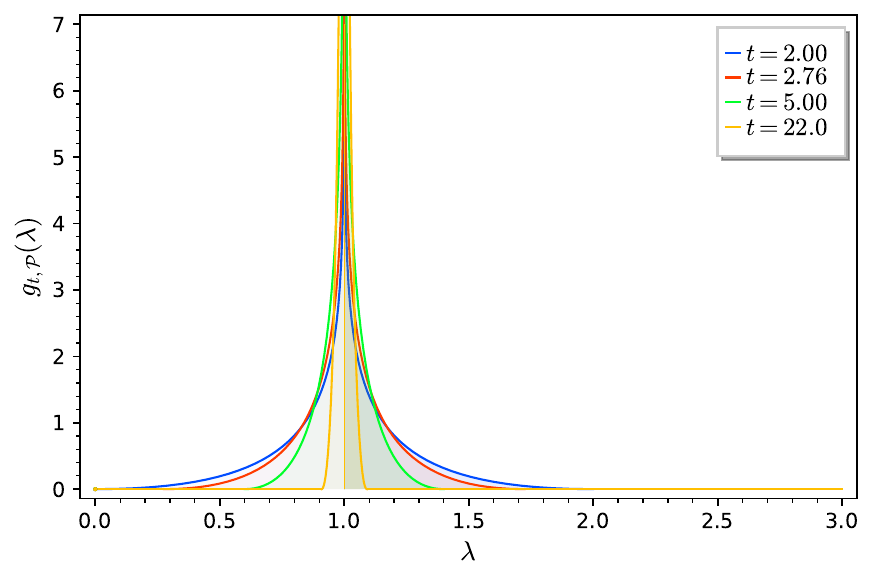}
\hfill\mbox{}
\caption{The graph of the limit density $g_{t,\cP}(\lambda)$ for $t\ge 2$.
In the image on the left-hand side, $t = 2.76$,
whereas in the image on the right-hand side, the graph of $g_{t,\cP}(\lambda)$ is 
is shown for $t=2,2.76,5$ and $22$.}
\label{F:DerivgtPLambda276}
\end{figure}

We notice that the inner limit is being taken of $G_{t,p,h}(\lambda-\delta)-G_{t,p,h}(\lambda+\delta)$, and so this expression becomes
\begin{equation*}
    g_{t,\cP}(\lambda) = \lim_{\delta\to0}\frac{1}{2\delta}
    \big(G_{t,\cP}(\lambda-\delta) - G_{t,\cP}(\lambda+\delta)\big),
\end{equation*}
where we have dropped the subscript $h$ for the density function
since the right hand side is now independent of~$h$. 
If the partial derivative of $G_{t,\cP}(\lambda)$ with respect to $\lambda$ exists, then
\begin{equation*}
    g_{t,\cP}(\lambda) = -\frac{\partial}{\partial\lambda}G_{t,\cP}(\lambda).
\end{equation*}
If the partial derivative doesn't exist at $\lambda_0$, but exists on some interval centered at $\lambda_0$, then one can see that
\begin{equation*}
    g_{t,\cP}(\lambda_0) = -\frac{1}{2}\Big(\lim_{\lambda\to\lambda_0^-}G_{t,\cP}(\lambda) + \lim_{\lambda\to\lambda_0^+}G_{t,\cP}(\lambda)\Big).
\end{equation*}
The graph of $g_{t,\cP}(\lambda)$ for a few values of $t\ge 2$ is shown in 
Figure~\ref{F:DerivgtPLambda276}. 

\section{Distribution of Rational Functions Modulo a Prime}\label{SectionDistributionRationalFunctions}

The key technical tool in this paper is the Riemann Hypothesis for curves over a finite field 
(see \cite{Bomb1966,Perel1963,Weil1948}) in the following form.
\begin{lemma}[\protect{\cite[Lemma 2.1]{Zah2003}}]\label{L:WeilBound}
    Let $p$ be a prime number and $\F_p$ the field with $p$ elements. Let $\psi$ be a non-trivial character of the additive group of $\F_p$ and let $R(x)$ be a non-constant rational function. Then
    \begin{equation*}
        \sum_{x\in\F_p} \psi\big(R(x)\big) \ll \sqrt{p},
    \end{equation*}
    where the sum excludes poles of $R(x)$, and the implied constant depends on the degrees of the numerator and the denominator of $R(x)$.
\end{lemma}

Next, we need an estimate for the number of points on a multidimensional modular curve. 

\begin{lemma}\label{LemmaNIJJ}
Let $p\ge 2$ be prime, and let $d\ge 1$ be integer.
Let $\cR_p$ be the set of all tuples $\big(r_1(x),\dots,r_d(x)\big)$ of reduced rational functions over $\F_p$ such that:
\begin{enumerate}
    \item For all $j\in\{1,\dots,d\}$, the degree of the numerator of $r_j$ is at most one.
    \item For all $j\in\{1,\dots,d\}$, the degree of the denominator of $r_j$ is equal to one.
    \item Where $a, b, c_1, \dots, c_d \in \F_p$,
        \begin{equation*}
            a + bx + c_1 r_1(x) + \cdots + c_d r_d(x) = 0
        \end{equation*}
        for all $x\in\F_p$ only if $a=b=c_1=\cdots=c_d=0$.
\end{enumerate}
Let $(r_1(x),\dots,r_d(x))\in\cR_p$ be fixed.
Then, where $\cJ$, $\cJ_1$, \dots, $\cJ_d$ are intervals in $\F_p$,
\begin{equation*}
\begin{split}
    \#\big\{ x \in \F_p : 
    \ x\in\cJ,\ r_k(x)\in \cJ_k, \text{ for $1\le k\le d$}\big\} 
    = \frac{\abs{\cJ}\abs{\cJ_1}\cdots\abs{\cJ_d}}{p^{d}} + O\left(p^{1/2}\log^{d+1} p\right).
\end{split}
\end{equation*}
\end{lemma}

\begin{proof}
 We will define the \emph{complete sum}
 \begin{equation*}
     S(a, b_1, \dots, b_d) := \sum_{x\pmod{p}} e\left(\frac{ax + b_1r_1(x) + \cdots + b_dr_d(x)}{p}\right),
 \end{equation*}
 and the \emph{incomplete sum}
 \begin{equation*}
     S_{\cJ}(a, b_1, \dots, b_d) := \sum_{x\in\cJ} e\left(\frac{ax + b_1r_1(x) + \cdots + b_dr_d(x)}{p}\right),
 \end{equation*}
 for some subinterval $\cJ$ of integers in $[0,p-1]$,
 where, in both the above sums and in the ones that follow, 
 we omit the poles of the rational functions $r_1, \dots, r_d$.

 We start off with the bound on the complete sum coming from Lemma~\ref{L:WeilBound}. 
 Provided that the $b_j$ aren't all~$0$,
\begin{equation*}
    S(a,b_1,\dots,b_d) \ll_d \sqrt{p}\,.
\end{equation*}

We now use this to find a bound on our incomplete sum $S_{\cJ}(a,b_1,\dots,b_d)$. 
With $\chi_{\cJ}(x)$ denoting the characteristic function
of the interval $\cJ$, this equals
\begin{align*}
    &\phantom{=}\ \sum_{x\pmod{p}} e\left(\frac{ax + b_1r_1(x) + \cdots + b_dr_d(x)}{p}\right) \chi_{\cJ}(x) \\
    &= \sum_{x\pmod{p}} e\left(\frac{ax + b_1r_1(x) + \cdots + b_dr_d(x)}{p}\right) \sum_{y\in\cJ}\frac{1}{p}\sum_{t\pmod{p}}e\left(\frac{t(y-x)}{p}\right) \\
    &=\frac{1}{p}\sum_{t\pmod{p}} \sum_{y\in\cJ} e\left(\frac{ty}{p}\right) S(a-t,b_1,\dots,b_d)\,.
\end{align*}

The sum over $y$ is a geometric progression
that can be estimated precisely:
\begin{equation*}
    \sum_{y\in\cJ} e\left(\frac{ty}{p}\right) \ll \min\left\{\abs{\cJ}, \frac{1}{2\norm{t/p}}\right\}.
\end{equation*}

It then follows that, as long as the $b_j$ aren't all 0, we find that the incomplete sum is bounded by
\begin{align*}
 \frac{1}{p}\sum_{t=1}^{p-1}\frac{1}{2\norm{t/p}}\sqrt{p} + \frac{1}{p}\abs{\cJ}\sqrt{p} 
    \ll_d \frac{1}{p}\sum_{t=1}^{(p-1)/2}\frac{p}{t}\sqrt{p} + \frac{1}{p}\abs{\cJ}\sqrt{p} 
    \ll_d \sqrt{p}\log p.
\end{align*}

Now, the set $\{x\in\cJ : r_1(x)\in\cJ_1, \dots, r_d(x)\in\cJ_d\}$ has cardinality
\begin{align*}
    &= \sum_{x\in\cJ} \begin{cases}
        1, & \text{if } r_1(x)\in\cJ_1, \dots, r_d(x)\in\cJ_d; \\
        0, & \text{else};
    \end{cases} \\
    &= \sum_{x\in\cJ} \left[\sum_{y_1\in\cJ_1} \frac{1}{p}\sum_{t_1\pmod{p}}e\left(\frac{t_1\big(y_1-r_1(x)\big)}{p}\right)\right] \cdots \left[\sum_{y_d\in\cJ_d} \frac{1}{p}\sum_{t_d\pmod{p}}e\left(\frac{t_d\big(y_d-r_d(x)\big)}{p}\right)\right] \\
    &= \frac{1}{p^d} \sum_{t_1\pmod{p}}\cdots\sum_{t_d\pmod{p}} \left(\sum_{y_1\in\cJ_1} e\left(\frac{t_1 y_1}{p}\right)\right) \cdots \left(\sum_{y_d\in\cJ_d} e\left(\frac{t_d y_d}{p}\right)\right) S_{\cJ}(0,-t_1,\dots,-t_d)\,.
\end{align*}

In the terms where $t_1=\cdots=t_d=0$, the sums are computed trivially to find our main term
\begin{equation*}
    \frac{\big(\abs{\cJ}-\abs{\cS}\big)\abs{\cJ_1}\cdots\abs{\cJ_d}}{p^d} = \frac{\abs{\cJ}\abs{\cJ_1}\cdots\abs{\cJ_d}}{p^d} + O_d(1),
\end{equation*}
where $\cS$ is the set of poles of the $r_j(x)$ appearing in $\cJ$.

The remaining terms contribute
\begin{equation*}
    \frac{1}{p^d} \sum_{(t_1, \dots, t_d) \neq (0, \dots, 0)} \prod_{j=1}^d \left[\sum_{y_j\in\cJ_j} e\left(\frac{t_j y_j}{p}\right)\right] S_{\cJ}(0,-t_1,\dots,-t_d).
\end{equation*}

As before, the sums over $y_j$ are bounded above by
\begin{equation*}
    O_d\left(\min\left\{\abs{\cJ_j}, \frac{1}{2\norm{t_j/p}}\right\}\right),
\end{equation*}
and our incomplete sum is bounded above by
\begin{equation*}
    O_d\big(\sqrt{p}\log p\big).
\end{equation*}

Consider the terms where $k$, with $0\leq k \leq d-1$, of the $t_j$'s are fixed at zero, and the remaining $t_j$ sum over the non-zero terms. There are $\binom{d}{k}$ choices of $t_j$ to fix at zero, and the contribution from each choice is bounded above by
\begin{align*}
    \sqrt{p}\log p \frac{1}{p^d} p^{k} \prod_{j=1}^{d-k} \sum_{t_j=1}^{p-1} \frac{1}{2\norm{t_j/p}} &\ll_d \sqrt{p}\log p \frac{1}{p^{d-k}} \prod_{j=1}^{d-k}
 \sum_{t_j=1}^{(p-1)/2} \frac{p}{t_j} \\
    &\ll_d \sqrt{p}\log p \prod_{j=1}^{d-k}\sum_{t_j=1}^{(p-1)/2}\frac{1}{t_j} \\
    &\ll_d \sqrt{p}\log^{d-k+1}p.
\end{align*}

Summing over all of these contributions, we arrive at our error term
\begin{equation*}
    \ll_d \sqrt{p}\log^{d+1} p,  
\end{equation*}
which completes the proof of the lemma.
\end{proof}

We need the following result, which will only be applied to the domain $\Omega(t,\lambda)$ given in Theorem~\ref{Thm1.1}.

\begin{corollary}\label{C:RationalFunctionsInDomain}
    Let $\Omega$ be a 
    domain inside the hypercube $[-1/2, 1/2]^{d+1}$ with a $C^1$ boundary and let $(r_1(x), \dots, r_d(x))$ be a tuple which satisfies the conditions of Lemma~\ref{LemmaNIJJ}. Consider the function 
\begin{equation*}
  \br: \F_p \to [-1/2, 1/2]^{d+1},\quad \br(x):=\big([x]/p, [r_1(x)]/p, \dots, [r_d(x)]/p\big),    
\end{equation*}
where $[x]$ is the representative of $x$ in $[-(p-1)/2, (p-1)/2]$. Let $N(\br, \Omega)$ count the number of $x\in\F_p$ such that $\br(x)\in\Omega$. Then,
    \begin{equation*}
        N(\br, \Omega) = p\mu(\Omega) + O_{d,\Omega}\left(p^{(2d+3)/(2d+4)}\log^{(d+1)/(d+2)}p\right),
    \end{equation*}
    where $\mu$ is the Lebesgue measure in $\R^{d+1}$.
\end{corollary}

\begin{proof}
    We split up the cube $[-1/2, 1/2]^{d+1}$ into smaller cubes of side length $1/L$, where $L$ is a positive integer. By Lemma~\ref{LemmaNIJJ}, for any such cube $\cJ$, the number of $x\in\F_p$ such that $\br(x) \in \cJ$ is
    \begin{equation}\label{E:LemmaNIJJResult}
        \frac{p}{L^{d+1}} + O_d\left(p^{1/2} \log^{d+1}p\right).
    \end{equation}

    Let $D(L)$ be the union of such cubes which are contained completely in $\Omega$, and let $E(L)$ be the union of such cubes which have a non-empty intersection with $\Omega$. Therefore,
    \begin{equation*}
        D(L) \subseteq \Omega \subseteq E(L).
    \end{equation*}

    For any point inside $E(L)\backslash D(L)$, there will be a point on $\partial\Omega$ at a distance at most $\sqrt{d+1}/L$ away. Thus, we can say that $E(L)\backslash D(L)$ is contained in the set of all points of distance at most $\sqrt{d+1}/L$ from~$\partial\Omega$. This set has measure $O_{d,\Omega}(1/L)$ (see \cite{Weyl1939}), and consequently 
$\mu(E(L)\backslash D(L))$ is also $O_{d,\Omega}(1/L)$.

    From this, we see that
    \begin{equation*}
        \mu(D(L)) = \mu(\Omega) - \mu\big(\Omega\backslash D(L)\big) = \mu(\Omega) + O_{d,\Omega}(1/L),
    \end{equation*}
    and similarly
    \begin{equation*}
        \mu\big(E(L)\big) = \mu(\Omega) + O_{d,\Omega}(1/L).
    \end{equation*}
    Hence, $D(L)$ and $E(L)$ are both unions of $L^{d+1}\mu(\Omega) + O_{d,\Omega}(L^{d})$ cubes of side length $1/L$. From this and \eqref{E:LemmaNIJJResult}, we find
    \begin{align*}
        \# \{x\in\F_p : \br(x)\in D(L)\}\big| &= \bigl(L^{d+1}\mu(\Omega) + O_{d,\Omega}(L^{d})\bigr)\bigl(p/L^{d+1}+O_d(p^{1/2}\log^{d+1}p)\bigr) \\
        &= p\mu(\Omega) + O_{d,\Omega}\big(L^{d+1}p^{1/2}\log^{d+1}p + p/L\big)\,.
    \end{align*}
    Likewise,
    \begin{equation*}
        \# \{x\in\F_p : \br(x)\in E(L)\}\big| 
        = p\mu(\Omega) + O_{d,\Omega}\big(L^{d+1}p^{1/2}\log^{d+1}p + p/L\big)\,,
    \end{equation*}
    therefore
    \begin{equation*}
        N(\br,\Omega) = p\mu(\Omega) + O_{d,\Omega}\big(L^{d+1}p^{1/2}\log^{d+1}p + p/L\big)\,.
    \end{equation*}
    Choosing
    \begin{equation*}
        L = \floor{p^{1/(2d+4)}\log^{-(d+1)/(d+2)}p},
    \end{equation*}
    we obtain our result.
\end{proof}

\section{Gap Distribution of Angles. Proof of Theorem~\ref{Thm1.1}}\label{SectionGapDistribution}

We now use results from Section~\ref{SectionDistributionRationalFunctions} to study the gap-distribution functions. 
We label the points in $\cA(p,h)$ as $(r(m), m)$, where 
\begin{equation}\label{eqDefr}
    r(m):=m(1-hm)^{-1} \pmod p
\end{equation}
is a rational function modulo $p$. We will look at the contribution from angles only in the top half of the square (laying above the $x$-axis). We will see the the bottom half of the square will give the same contribution to $G_{t,p}(\lambda)$. For any point $(r(m), m)$, we will want to know what is the next point seen by the observer (in particular, on which line does it lie). We may not be able to say with certainty which line it lies on, but for fixed $t$, we can say that the next point seen lies within a bounded number of lines away from $m$.

For any integer $m$ between $0$ and $J$, the line between $(-J,m)$ and $(J, m+D)$ has intersection with the $x$-axis at
\begin{equation*}
    \left(-\frac{2Jm}{D}-J, 0\right).
\end{equation*}

If our observer is to the left of this point, then the observer will see the entire row $y=m$ before it starts seeing the row $y=m+D$. If
\begin{equation*}
    \frac{2}{D} < t \leq \frac{2}{D-1},
\end{equation*}
then once $J$ is sufficiently large, the above holds true for all $m\leq J$. We say that we have \textit{interference} from up to $D-1$ lines away, but no more. From now on, we assume that $t$ lies in this range.

The entire row $y=m+1$ is seen by the observer before any of the row $y=m+1+D$ is seen. Therefore, the next point seen after $(r(m), m)$ cannot lay on the line $y=m+1+D$, or any line above. Moreover, the entire row $m-D$ is seen before any of the line $y=m$, so the next point seen after $(r(m), m)$ cannot lay on the line $y=m-D$, or any line below.

To count whether the angle between $(r(m), m)$ and the next point seen is larger than $\lambda\Delta_{av}$ or not, we start off by drawing two rays: the ray from $(-tJ^2, 0)$ 
to $(r(m), m)$, and the ray from $(-tJ^2, 0)$ which forms an angle $\lambda\Delta_{av}$ in the anti-clockwise direction from the first ray. If there are no other points in the region bounded by the two rays and the rectangle, we can say that the angle from $(r(m), m)$ to the next point seen is large enough to be counted.

We can count only the points satisfying $m>J^\eta$ for some $\eta\in(0,1)$, to be further specified later. (We will see that any $\eta\in[5/8, 3/4]$ will be sufficient for our purpose without worsening the final error term.) This will give us an error of $O(p^{\eta})$ when counting points, but will also allow us to have lower bound on $m$.

The first ray has equation
\begin{equation*}
    y=\frac{m}{r(m)+tJ^2}(x+tJ^2),
\end{equation*}
and so it intersects the line $y=m+j$ at the point
\begin{equation*}
    \left(r(m)+j\frac{tJ^2}{m} + O(J^{1-\eta}), m+j\right).
\end{equation*}

The first ray makes an angle of
\begin{equation*}
    \frac{m}{r(m)+tJ^2} + O(J^{-3})
\end{equation*}
with the $x$-axis. Its slope has a similar expresion since $h\ll J^{-1}$ implies $\tan h=h+O(J^{-3})$ and $\arctan h =h+O(J^{-3})$
It follows that the second ray will make an angle of
\begin{equation*}
    \frac{m}{r(m)+tJ^2} + \frac{\lambda}{tJ^2} + O(J^{-3})
\end{equation*}
with the $x$-axis. It follows that the second ray has equation
\begin{equation*}
    y=\left(\frac{m}{tJ^2+r(m)} + \frac{\lambda}{tJ^2} + O(J^{-3})\right)(x+tJ^2),
\end{equation*}
and so it intersects the line $y=m+j$ at some point of the form
\begin{equation*}
    \bigg( \Big( r(m) + (j-\lambda)\frac{tJ^2}{m}\Big) \big( 1 + O(J^{-\eta})\big), m+j\bigg) =
\bigg( r(m)+(j-\lambda)\frac{tJ^2}{m}+O(J^{2-2\eta}),m+j \bigg) .
\end{equation*}

As a consequence, we want to find $m$ such that
\begin{align}
    J^{\eta} <\ &m < J,\label{E:ineqm} \\
    -J < r(&m) < J,\label{E:ineqrm}
\end{align}
and either
\begin{equation}\label{E:ineqrmj1}
    -J < r(m+j) < r(m) + (j-\lambda)\frac{tJ^2}{m} + O(J^{2-2\eta}),
\end{equation}
or
\begin{equation}\label{E:ineqrmj2}
    r(m) + j\frac{tJ^2}{m} + O(J^{1-\eta}) < r(m+j) < J,
\end{equation}
with \eqref{E:ineqrmj1} and \eqref{E:ineqrmj2} 
holding for all integers $j$ satisfying $-D+1\leq j \leq D$, $j\neq0$.

We now wish to apply Corollary~\ref{C:RationalFunctionsInDomain} to the tuple $(r(m+(-D+1)), \dots, r(m+D))$. The rational function
$m \mapsto r(m+j)$ has a pole at $m$ with $(m+j)h \equiv 1 \pmod{p}$, according to~\eqref{eqDefr},
so as long as $2D<p$, each such rational function will have a different pole. Therefore, no linear combination of the rational functions will give a linear function, and so Corollary~\ref{C:RationalFunctionsInDomain} can be applied, with $d=2D$.

Dividing \eqref{E:ineqm}, \eqref{E:ineqrm}, \eqref{E:ineqrmj1}, and \eqref{E:ineqrmj2} by $p$, and then writing $x=m/p$ and \mbox{$y_j=r(m+j)/p$}, Corollary~\ref{C:RationalFunctionsInDomain} tells us that the number of $m$ satisfying the above inequalities is asymptotically equal to $p\mu(\Omega')$, where $\Omega'\subseteq[-1/2, 1/2]^{2D+1}$ is the set of points $(x, y_{-D+1}, \dots, y_{D})$ satisfying
\begin{align}
    x&\in\left[O(p^{\eta-1}), \frac{1}{2}+O(p^{-1})\right], \label{E:ineqxerr} \\ 
    y_0 &\in \left[-\frac{1}{2}+O(p^{-1}), \frac{1}{2} + O(p^{-1})\right], \label{E:ineqy0err} \\
    y_j &\in \left[-\frac{1}{2}, \frac{1}{2}\right] \setminus \left[y_0+(j-\lambda)\frac{t}{4x} + O(p^{1-2\eta}), y_0 + j\frac{t}{4x} + O(p^{-\eta})\right], \label{E:ineqyjerr}
\end{align}
with \eqref{E:ineqyjerr} holding for all integers $j\neq0$ such that $-D+1\leq j \leq D$.

We show that $\mu(\Omega')$ is very close to $\mu(\Omega)$, where $\Omega\subseteq[-1/2, 1/2]^{2D+1}$ is the set of points $(x, y_{-D+1}, \dots, y_{D})$ satisfying
\begin{align}
    x&\in\left[0, \frac{1}{2}\right], \label{E:ineqx} \\ 
    y_0 &\in \left[-\frac{1}{2}, \frac{1}{2}\right], \label{E:ineqy0} \\
    y_j &\in \left[-\frac{1}{2}, \frac{1}{2}\right] \setminus \left[y_0+(j-\lambda)\frac{t}{4x}, y_0 + j\frac{t}{4x}\right], \label{E:ineqyj}
\end{align}
with \eqref{E:ineqyj} holding for all integers $j\neq0$ such that $-D+1\leq j \leq D$.

Let $\Omega'_{k}$ be the region satisfying firstly \eqref{E:ineqxerr}, as well as \eqref{E:ineqy0err} and \eqref{E:ineqyjerr} for $j<k$, whilst satisfying \eqref{E:ineqy0} and \eqref{E:ineqyj} for $j\geq k$. Note that $\Omega'_{D+1}=\Omega'$. Then, $\mu(\Omega'_{k}) - \mu(\Omega'_{k+1}) = O(p^{1-2\eta})$. Also, $\mu(\Omega) = \mu(\Omega'_{-D+1}) + O(p^{\eta-1})$. Therefore,
\begin{equation*}
    \mu(\Omega') = \mu(\Omega) + O_D(p^{1-2\eta} + p^{\eta-1}).
\end{equation*}

It follows that the number of $m$ satisfying \eqref{E:ineqm}, \eqref{E:ineqrm}, \eqref{E:ineqrmj1}, and \eqref{E:ineqrmj2} is equal to
\begin{equation*}
    p\mu(\Omega) + O\left(p^{2-2\eta} + p^{\eta} 
    + p^{(4D+3)/(4D+4)}\log^{(2D+1)/(2D+2)}p\right),
\end{equation*}
and consequently, dividing by $p$, we obtain the contribution to $G_{t,p}(\lambda)$ from the top half of the square to be
\begin{equation*}
    \mu(\Omega) + O\left(p^{1-2\eta} + p^{\eta-1} 
    + p^{-1/(4D+4)}\log^{(2D+1)/(2D+2)}p\right).
\end{equation*}
Finally, we can take any $\eta\in\big[ 1/2+1/(8D+8), 1-1/(4D+4)\big]$ (say, $\eta=11/16$) and the first two error terms will be absorbed by the last error term.

To deal with the lower half of the square, we notice that if we reflect this rectangle about the $x$-axis, the angles between consecutive points are preserved. Our analysis of this rectangle is identical as before, simply replacing $r(m)$ by $r(-m)$, since the tuple $\big(r(-(m-D+1)), \dots, r(-(m+D)\big)$ also satisfies the conditions of Corollary~\ref{C:RationalFunctionsInDomain}. Therefore, the contribution to $G_{t,p}(\lambda)$ from the lower half of the square is identical to that from the upper half, so we obtain
\begin{equation*}
    G_{t,p}(\lambda) = 2\mu(\Omega) + O\left(
    p^{-1/(4D+4)}\log^{(2D+1)/(2D+2)}p\right).
\end{equation*}
Taking $p\to\infty$, the error term vanishes and we are left with
\begin{equation*}
    G_{t,\cP}(\lambda) = 2\mu(\Omega).
\end{equation*}
This completes the proof of Theorem~\ref{Thm1.1}.
\hfill\qed

\vspace*{10mm}
\section*{Appendix}
\renewcommand{\thesubsection}{\textbf{A\hspace{1pt}\arabic{subsection}}}
\setcounter{subsection}{0}
\subsection{The expressions for the gap distribution and the density functions for \texorpdfstring{$1\le t<2$}{t in [1,2]}}
Direct calculations now lead to the following explicit expressions of $G_{t,\cP}$ and $g_{t,\cP}$ on the specified intervals. 
Consider
\begin{align*}
H_1(t,\lambda):= & \frac{1}{2}\left(2-\lambda t^2 -2\lambda t\log\left(\frac{2}{t}\right)-t(1-\lambda)\log(1-\lambda) 
+t(1+\lambda)\log(1+\lambda)\right); \hspace{24mm} \\[2mm]
H_2(t,\lambda):= & \frac{1}{8}\left(4+t^2-2\lambda t^2+\lambda^2 t^2 + 4t(1-\lambda)
\log\left(\frac{2}{t(1-\lambda)}\right)\right); \\[2mm]
H_3(t,\lambda):= & \frac{1}{48}\Big(8+12t+6t^2+4t^3-24\lambda t + 12\lambda t^2
-12\lambda t^3 -6\lambda^2t^2+9\lambda^2t^3
  \\
& \phantom{\frac{1}{48}\Big(} -2\lambda^3t^3 +24t(2-\lambda)
\log\left(\frac{2}{t}\right) 
   -24t(1-\lambda)\log(1-\lambda)
           -24t\log(2-\lambda) \Big) ; \\[2mm]
H_4(t,\lambda):= & \frac{1}{48}\Big( 8-12t+6t^2-t^3
+12\lambda t^2-6\lambda^2t^2 
    +24t(2-\lambda)\log\left(\frac{2}{t}\right)
    +24t(\lambda-1)\log(\lambda-1)\Big) ;  \\[2mm]
H_5 (t,\lambda):= &
\frac{1}{576}\Big( 144-80t-144t^2-12t^3+27t^4+8\lambda t+216\lambda t^2+54\lambda t^3
-54\lambda t^4
-72\lambda^2t^2-36\lambda^2t^3  \\
& \phantom{\frac{1}{576}\Big(} +36\lambda^2t^4 
+6\lambda^3t^3-10\lambda^3t^4+\lambda^4t^4 
-48t(5\lambda-8)\log\left(\frac{2}{t}\right)+48t(4-\lambda)\log(3-\lambda) \\ & \phantom{\frac{1}{576}\Big(} 
+288t(\lambda-1)\log(\lambda-1)\Big) ;\\[2mm]
H_6 (t,\lambda) := & \frac{1}{576}\Big(
144-416t-432t^2-108t^3-13t^4+152\lambda t
  +504\lambda t^2+198\lambda t^3+34\lambda t^4 
  -144\lambda^2t^2  \\ & 
  \phantom{\frac{1}{576}\Big(}
  -108\lambda^2t^3  -30\lambda^2t^4  
  +18\lambda^3t^3+10\lambda^3t^4-\lambda^4 t^4 
  + \frac{48t}{\lambda-1}+ 48t(8-5\lambda)
  \log\Big(\frac{2}{t(\lambda-1)}\Big)\Big)  ; \\[2mm]
H_7(t,\lambda) := & \frac{1}{48}\Big(32+12t+4t^3-24\lambda t
-12\lambda t^3-12\lambda^2t^2+9\lambda^2t^3
-2\lambda^3t^3 +24t(1-2\lambda)\log\left(\frac{2}{t}\right) \\
& \phantom{\frac{1}{48}\Big(}
-24t(1-\lambda)\log(1-\lambda)
-24t\log(2-\lambda)+24t(1+\lambda)\log(1+\lambda)\Big).
\end{align*}

\begin{remark}\label{RemarkA}
A straightforward calculation of $\mu(\Omega)$ then leads to the conclusion that:
\begin{multicols}{2}

(i) When $4/3\le t<2$,
\begin{equation*}
G_{t,\cP}(\lambda) = \begin{cases} 
H_1(t,\lambda) & \mbox{\rm if $\lambda \le \frac{2}{t}-1$;}\\
H_2(t,\lambda) & \mbox{\rm $\frac{2}{t}-1\le \lambda<2-\frac{2}{t}$;} \\
H_3(t,\lambda) & \mbox{\rm if $2-\frac{2}{t}\le \lambda<1$;}\\
H_4(t,\lambda) & \mbox{\rm if $1\le \lambda<3-\frac{2}{t}$;}\\
H_5(t,\lambda) & \mbox{\rm if $3-\frac{2}{t}\le \lambda<2$;}\\
H_6 (t,\lambda) & \mbox{\rm if $2\le \lambda<1+\frac{2}{t}$;}\\
0 & \mbox{\rm if $1+\frac{2}{t}\le \lambda$.   }
\end{cases}
\end{equation*}

(ii)    When $1\le t<4/3$, 
\begin{equation*}
G_{t,\cP}(\lambda) = \begin{cases}
H_1(t,\lambda) & 
   \mbox{\rm if $\lambda \leq 2-\frac{2}{t}$;}\\
H_7(t,\lambda) & 
\mbox{\rm if $2-\frac{2}{t}\leq \lambda < \frac{2}{t}-1$;}\\
H_3(t,\lambda) & \mbox{\rm if $\frac{2}{t}-1 \leq \lambda <1$;}\\
H_4(t,\lambda) & 
\mbox{\rm if $1\leq\lambda < 3-\frac{2}{t}$;}\\
H_5(t,\lambda) & 
\mbox{\rm if $3-\frac{2}{t}\leq\lambda <2$;} \\
H_6(t,\lambda) &
\mbox{\rm if $2\leq\lambda < 1+\frac{2}{t}$;} \\
0 & \mbox{\rm if $1+\frac{2}{t} \leq \lambda$.}
\end{cases}
\end{equation*}
\end{multicols}
\end{remark}

For $4/3\le t<2$ and then $1\le t<4/3$, the graphs shown in Figure~\ref{FigureGtPLambda145-112} 
represents two instances of the limit gap distribution function,
emphasizing the regions where the general expression changes.
One sees that these fit exactly with the corresponding $G^*_{t,q,1}(\lambda)$ calculated numerically
for the prime $q=8009$.
\begin{center}
\begin{figure}[htb]
\centering
\hfill
    \includegraphics[width=0.49\textwidth]{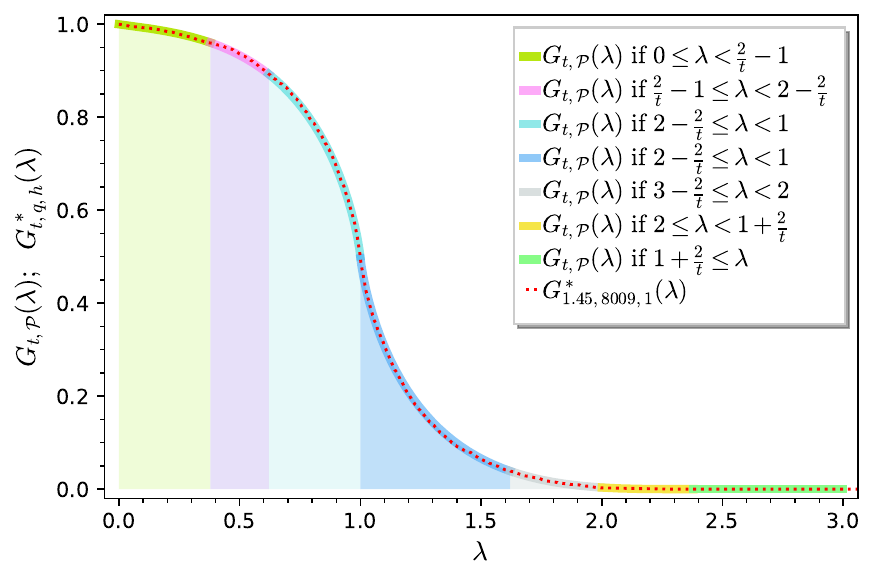}
\hfill\hfill
    \includegraphics[width=0.49\textwidth]{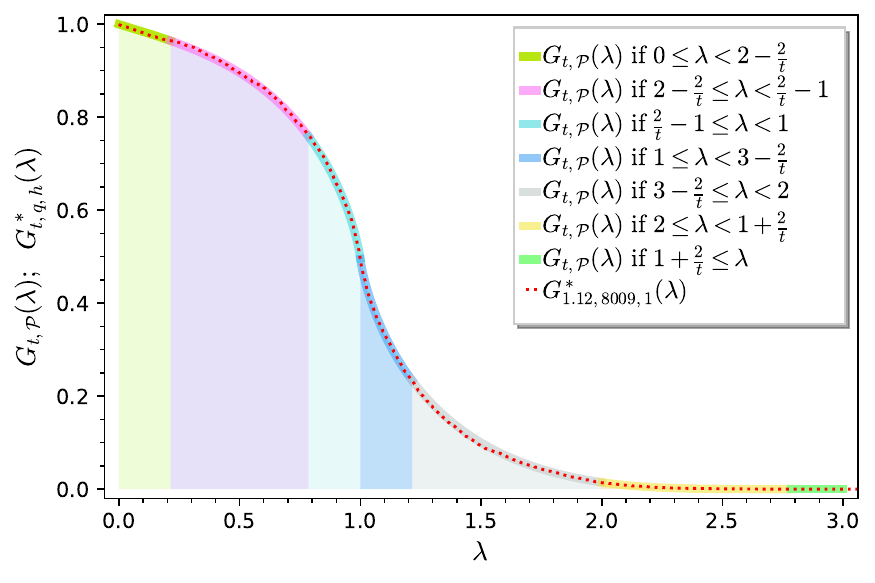}
\hfill
\caption{The graph of $G_{t,\cP}(\lambda)$ for $t=1.45$ (left) and $t=1.12$ (right).
In each case, 
the dotted graph of $G^*_{t,q,h}(\lambda)$ calculated numerically 
with $q=8009$ and $h=1$ is overlaid on top.
The shaded regions of the subgraphs indicate the regions where $G_{t,\cP}(\lambda)$ has different expressions.}
\label{FigureGtPLambda145-112}
\end{figure}
\end{center}

In this range of $t$, the density function formula is given
for every $\lambda \neq 1$ by
$g_{t,\cP}(\lambda) =-G^\prime_{t,\cP}(\lambda)$ and the graphs are shown in Figure~\ref{Figure-density145-112}.

\begin{center}
\begin{figure}[htb]
\centering
\hfill
    \includegraphics[width=0.49\textwidth]{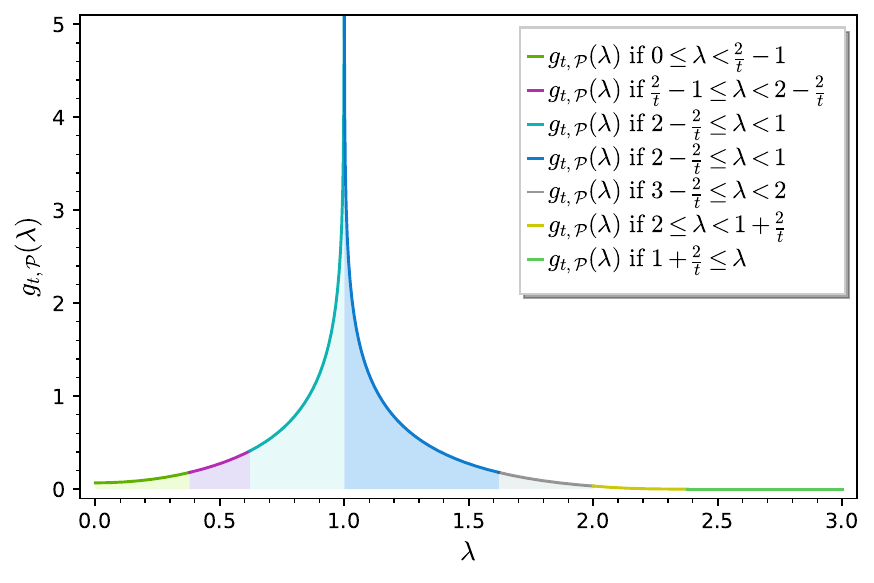}
\hfill\hfill
    \includegraphics[width=0.49\textwidth]{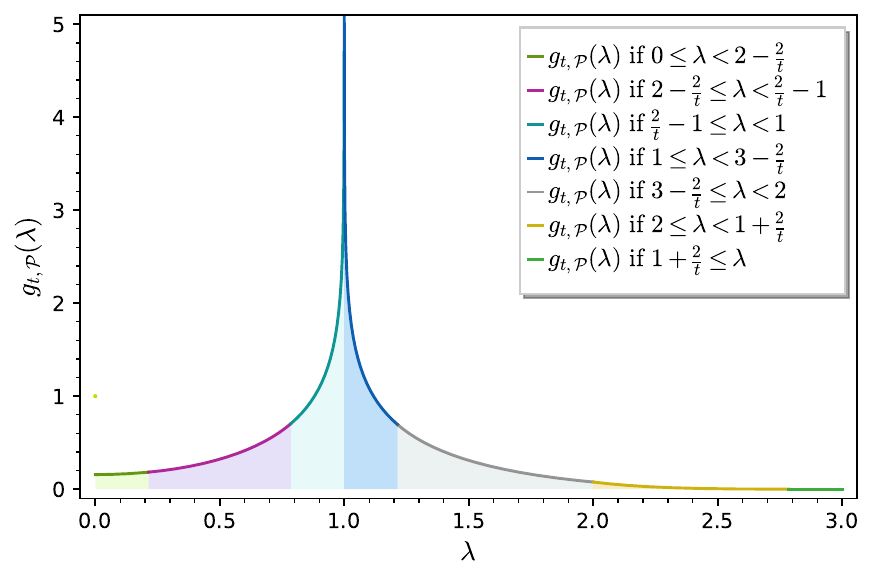}
\hfill
\caption{The graph of $g_{t,\cP}(\lambda)$ for $t=1.45$ (left) and $t=1.12$ (right).
}
\label{Figure-density145-112}
\end{figure}
\end{center}

\begin{center}
\begin{figure}[htb]
\centering
\hfill
    \includegraphics[width=0.59\textwidth]{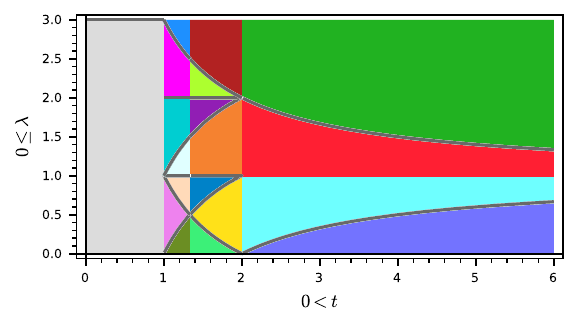}
\hfill
\mbox{}
\caption{The regions in which $G_{t,\cP}(\lambda)$ 
and $g_{t,\cP}(\lambda)$
change their expression for $t\ge 1$ and $\lambda\ge 0$.
}
\label{FigureDomains}
\end{figure}
\end{center}

\subsection{The tiles and the expressions of \texorpdfstring{$G_{t,\cP}(\lambda)$}{G(t,P}}
The curved polygonal regions bounded by the conditions imposed 
by the inequalities in Remark~\ref{RemarkA} are shown in Figure~\ref{FigureDomains}.
They tessellate the domain $[1,\infty)\times [0,\infty)$, and on each of them, say $\mathcal{D}$,
$G_{t,\cP}(\lambda)$ is given by the same expression for all $(t,\lambda)\in\mathcal{D}$.
Notice how the same formula holds for $G_{t,\cP}(\lambda)$ in neighbor domains situated 
one in the half planes $t>4/3$ and the other in $t<4/3$.
Furthermore, it can be noted that this fact is an extension of a pattern occurring for $t>2$, 
noting that $H_2(t,\lambda)$ appears also as the second case in the expression of 
$G_{t,\cP}(\lambda)$ given in Corollary~\ref{CorollaryG1} for $t\ge 2$.
It would be interesting to know how the pattern of the tiles and the corresponding formulas
continues for $t<1$.

\subsection{The neighbor-flips curves and their union}
We have pointed out in the introduction that it is experimentally observed 
that the $G^*_{t,q,h}(\lambda)$ are very different if $q$ is not prime, 
so the convergence of the sequence of gap distribution functions proved in Theorem~\ref{Thm1.1}
only occurs when $q$ tends to infinity through a sequence of prime numbers. 
Here, let's also observe that the reason behind
the existence of such different gap distribution functions
lies in the highly different form of the curves when $q$ is not prime.

For integers $q\ge 2$ and $0\le h\le q-1$, we consider the 
\textit{neighbor-flips} modular curves defined by
\begin{equation*}
    \NF(q,h) := \big\{\big(n^{-1}, (n+h)^{-1}\big) : n\in\Z/q\Z,\ \gcd(n,q)=\gcd(n+h,q)=1\big\}
    \subseteq [0,q]^2 .
\end{equation*}
This is the analogue of $\cA(q,h)$ and, geometrically,
it has the points arranged in the same way, 
except that they are not translated by $(-J,-J)$, so that the origin is the bottom left corner of the square $[0,q-1]^2$ that contains it, and not its center.

  \begin{figure}[htb]
\centering
\hfill
   \includegraphics[angle=0,width=0.24\textwidth]{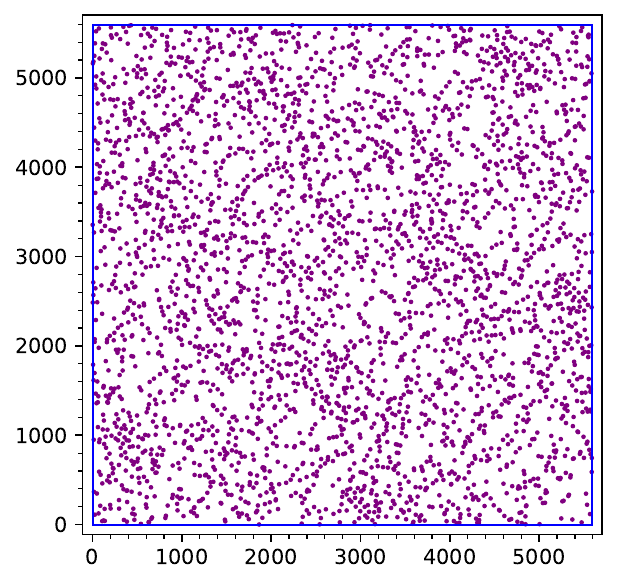}
\hfill\hfill
   \includegraphics[angle=0,width=0.24\textwidth]{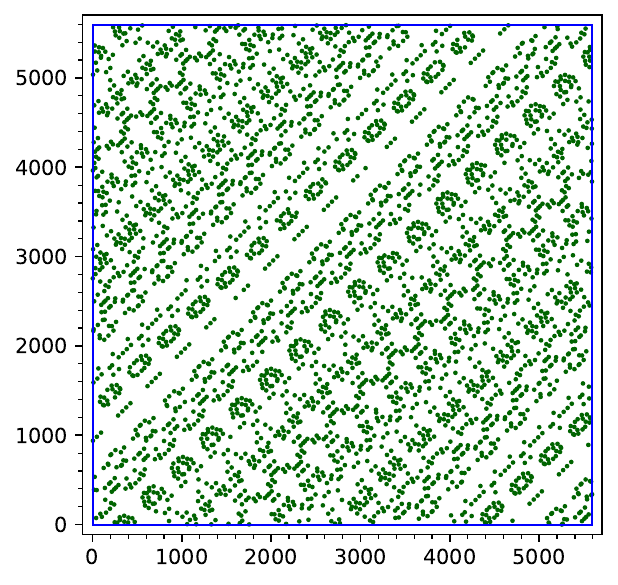}
\hfill\hfill
   \includegraphics[angle=0,width=0.24\textwidth]{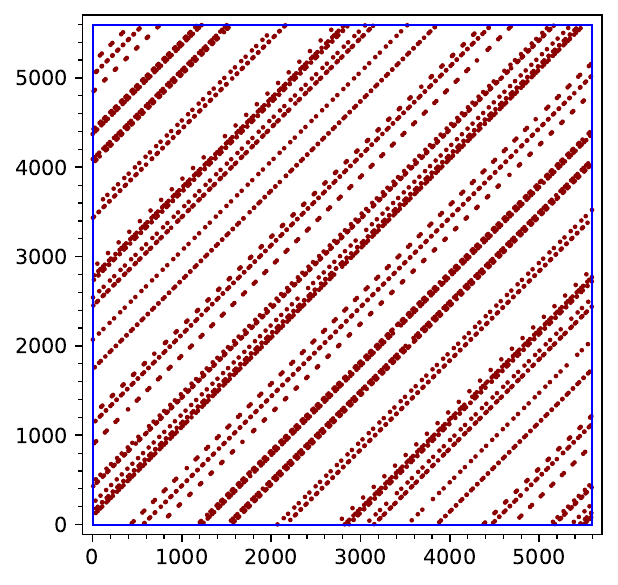}
\hfill\hfill
   \includegraphics[angle=0,width=0.24\textwidth]{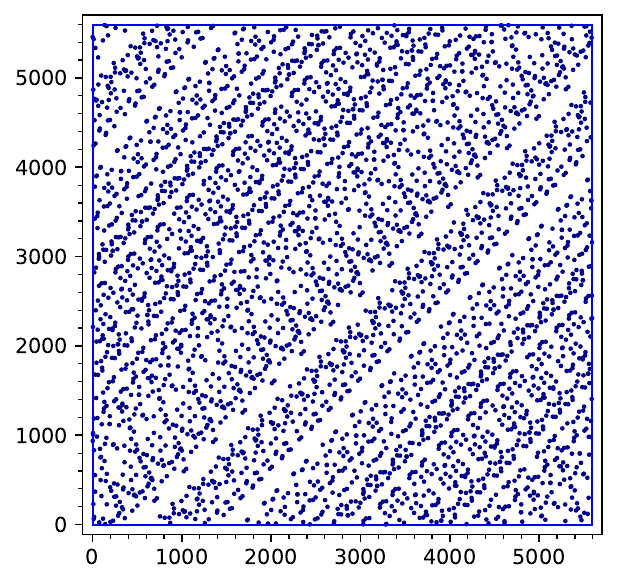}   
\hfill\mbox{}
\caption{The neighbor-flips curves $\NF(5593,h)$ 
for $q=5593=7\cdot 17\cdot 47$ and $h=4,34,141,289$.}
 \label{FigureNF5593}
 \end{figure}
 
Even for not necessarily highly composite moduli $q$, one notes
distinguished patterns featured by $\NF(q,h)$ when $q$
varies, or if $h$ varies for the same $q$
(see Figure~\ref{FigureNF5593} for such patterns).
Then, in cases like this, it might be possible to measure 
the size of the point-free regions of the curve and quantify 
the size of the large gaps that the observer sees from a certain distance.
It would be interesting to investigate whether there exist sequences of composite 
$q$'s for which a different limit gap distribution function exist.

Another distinction between the prime and non-prime moduli case can be seen in the unions 
\begin{equation*}
\NF(q):=\bigcup_{h=0}^{q-1}\NF(q,h).     
\end{equation*}
If $q$ is prime, $\NF(q)$ covers the square $[0,q-1]^2$ 
except for the lines $x=0$ and {\color{red}$y=q-h$}, 
while in the square remain uncovered periodic tubes of 
different widths depending on $q$, if $q$ is non-prime
(see the difference in Figure~\ref{FigureNFal}, where, for $q=100$ and $q=101$,
each $\NF(q,h)$ is drawn in a different color).

Let us also note that $\NF(q)$ is partitioned as a set by the curves 
$\NF(q,h)$ for $0\le h\le q-1$ because, according to the definition, 
there are no common points in $\NF(q,h)$ for distinct $h$'s.

\begin{center}
\begin{figure}[htb]
\centering
\hfill
   \includegraphics[angle=0,width=0.39\textwidth]{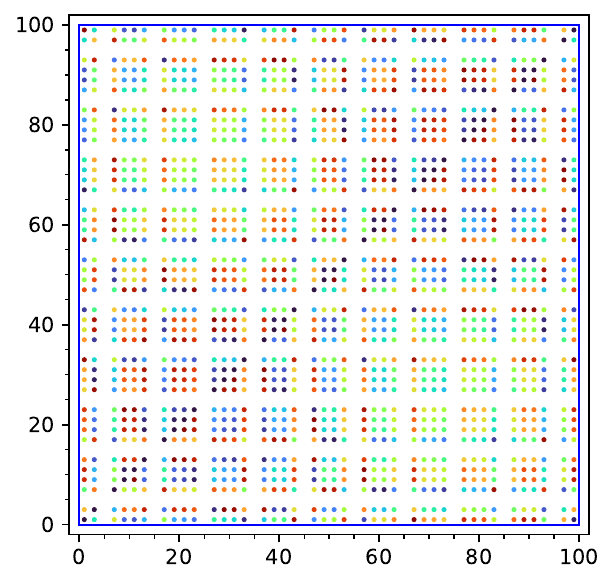}
\hfill
   \includegraphics[angle=0,width=0.39\textwidth]{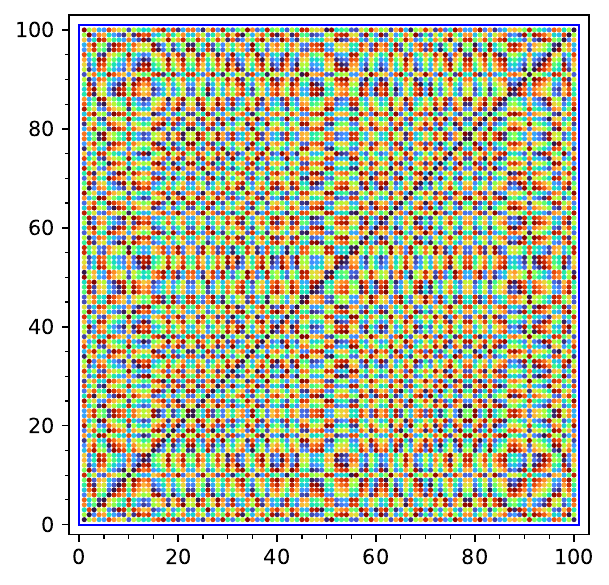}
\hfill\mbox{}
\caption{The unions $\NF(100)$ (left) and $\NF(101)$ (right).
}
 \label{FigureNFal}
 \end{figure}
\end{center}

\begin{figure}[tb]
\centering
\hfill
   \includegraphics[angle=0,width=0.24\textwidth]{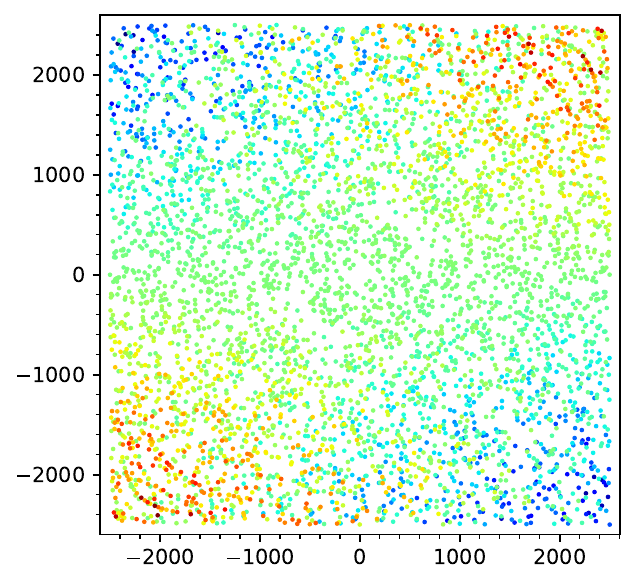}
\hfill\hfill
   \includegraphics[angle=0,width=0.24\textwidth]{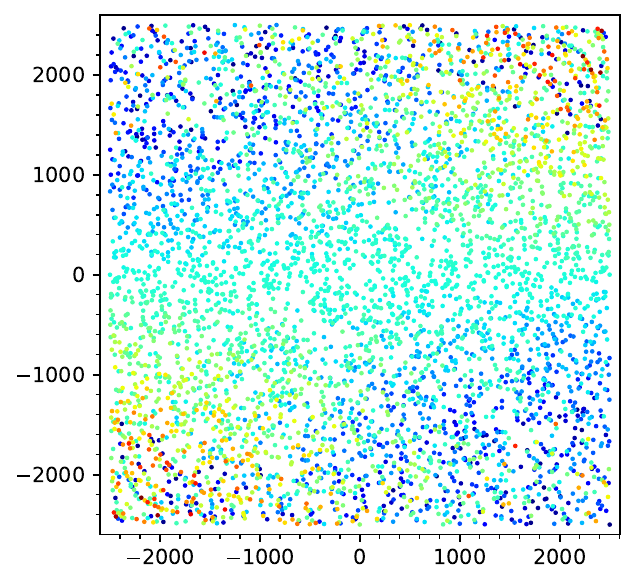}
\hfill\hfill
   \includegraphics[angle=0,width=0.24\textwidth]{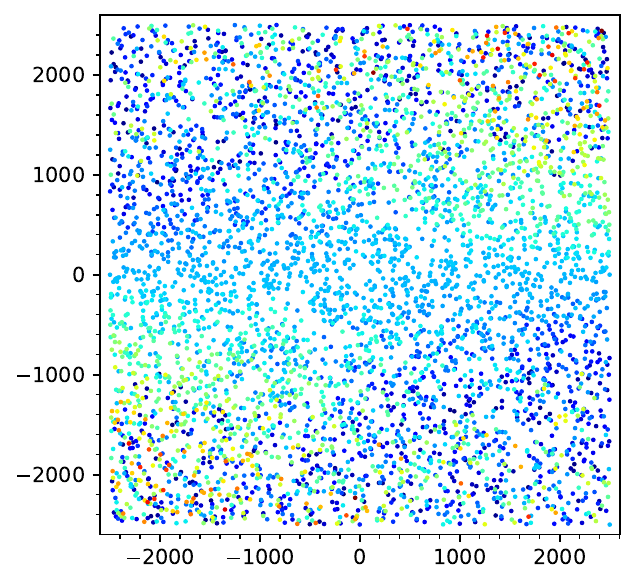}
\hfill\hfill
   \includegraphics[angle=0,width=0.24\textwidth]{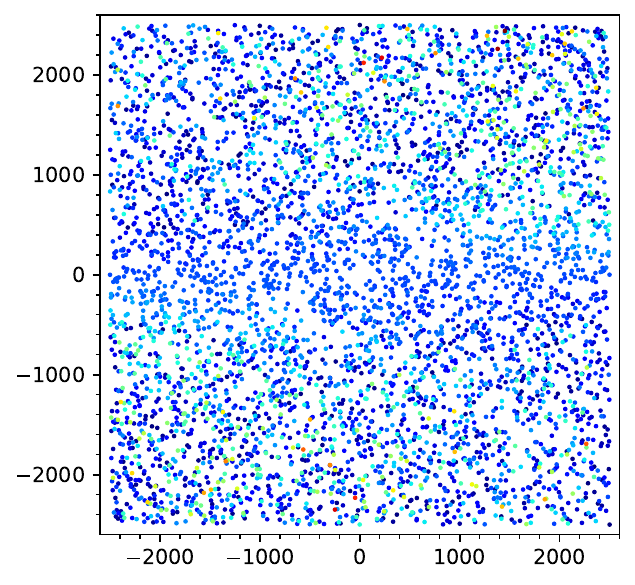}   
\hfill\mbox{}\\[2mm]
\hspace*{-3mm}
   \includegraphics[angle=0,width=0.24\textwidth]{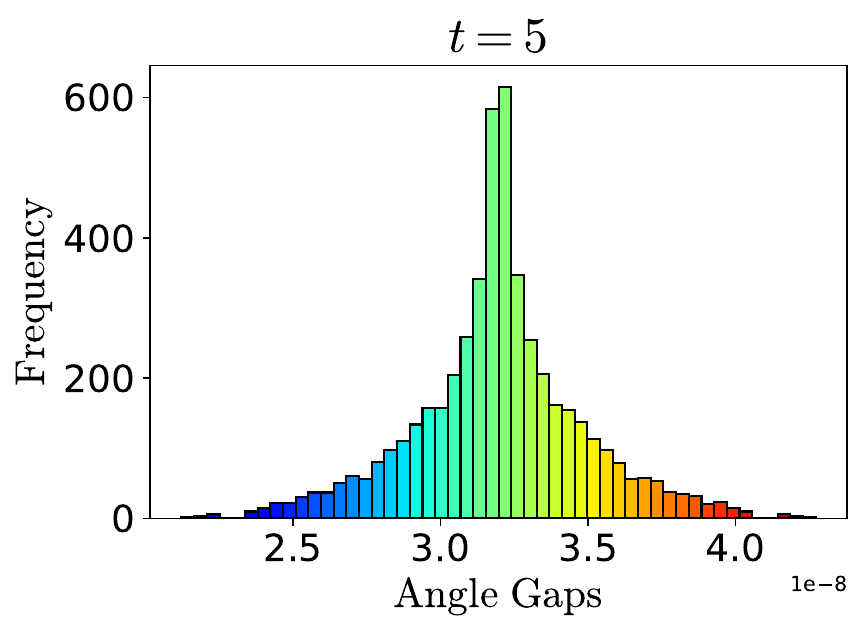}\
   \includegraphics[angle=0,width=0.24\textwidth]{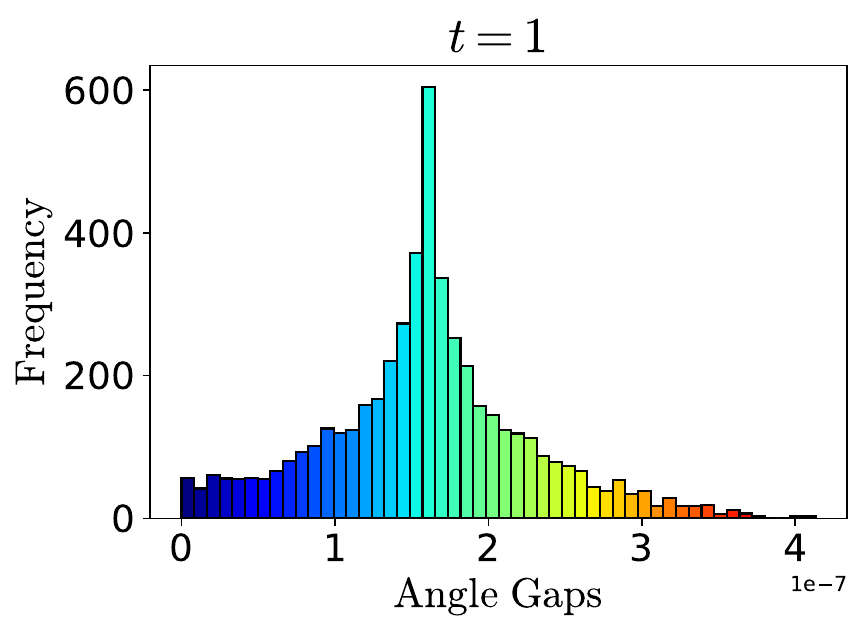}\ \ \
   \includegraphics[angle=0,width=0.24\textwidth]{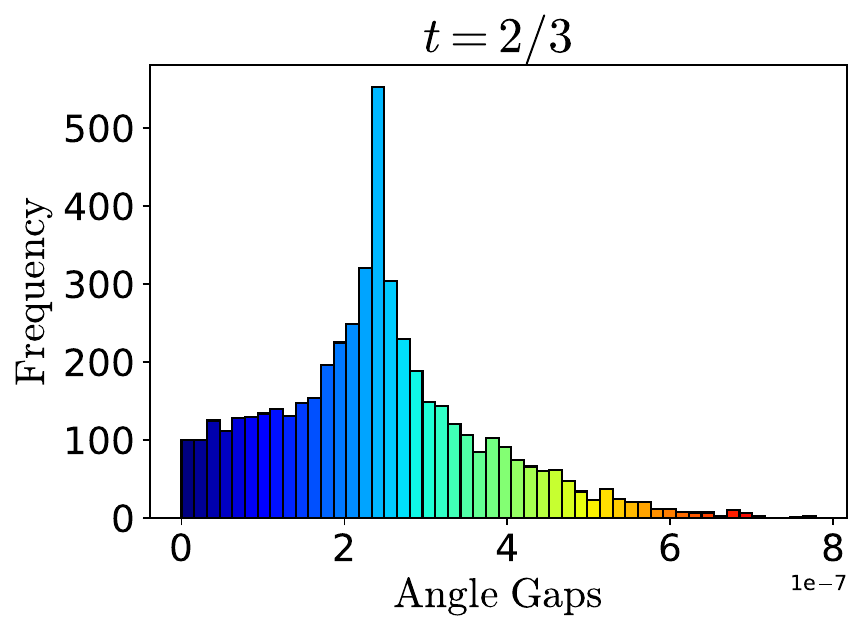}\
   \includegraphics[angle=0,width=0.24\textwidth]{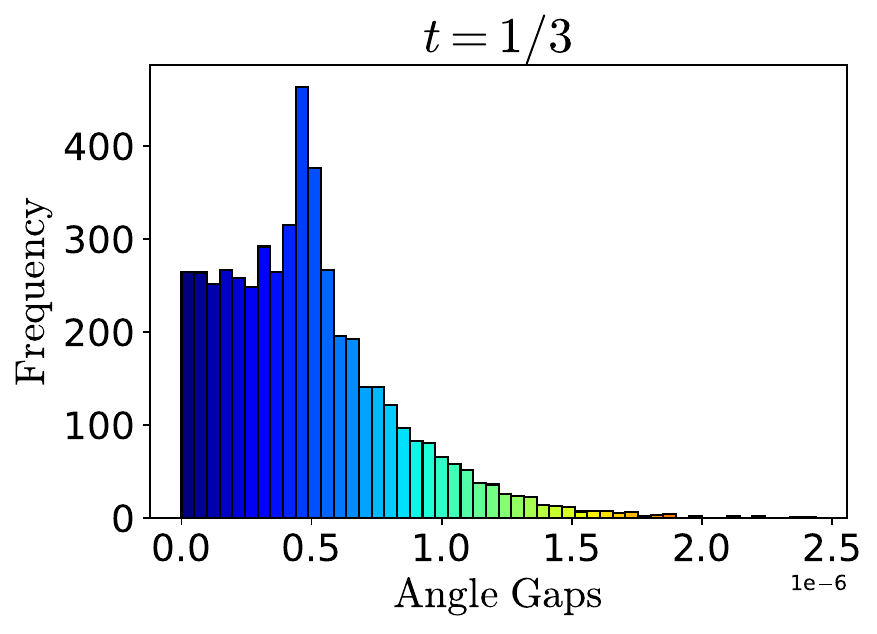} 
\caption{The curve $\cA(5003,1)$. 
The points are colored
by the size of the \textit{next-to point above gap} 
seen by the observer placed at the point of view
$(-tJ^2,0)$ for $t=5,1,2/3$ and $1/3$, respectively.
On the bottom line are shown the histograms
of the frequencies of the size of the gaps in the four cases above.}
 \label{FigureHistograms}
 \end{figure}

\subsection{
\texorpdfstring{$\cA(p,h)$ and $G^*_{t,p,h}(\lambda)$ for variable $t$}{A(p,h) and G*{t,p,h}(lambda) for variable t}
}
If $p$ is kept fixed,
depending on the size of $t>1/J$ (so that the point of view $(-tJ^2,0)$
remains outside of the smallest square that contains the curve $\cA(p,h)$)
the gaps between the angles towards the points seen by the observer
has a tendency to cluster, by their size, on certain regions of $\cA(p,h)$.
The shape of the clusters changes slowly if $t>2$ 
(when all points are seen by the observer in order as no interference occurs), and then,
when $t<2$ decreases, the clusters rapidly begin to blend together.

What happens is that when the observer is far away, the number of small angles 
is approximately equal to the number of large angles, their distribution being almost symmetrical. 
Then, when the observer gets closer and closer to the curve $\cA(p,h)$, the number of small angles seen overwhelms 
the number of large angles.
This can be observed in Figure~\ref{FigureHistograms}, where the points of 
$\cA(p,h)$ are colored, indicating the size of the angle 
that the observer sees towards the point that immediately follows in the counterclockwise rotation of the gaze.
There, the associated histograms indicate the frequency of the angles according to the position of the observer.


The phenomenon is also caught in a different way by the gap distribution function,
which appears to bend along the curve $e^{-\lambda}$ as $t$ converges down to $1/J$ 
(see Figure~\ref{FigureGtoRandom}).
Note that the later is the gap distribution function for a modular curve 
whose points are pseudo-randomly distributed.
This supports the common belief that a distant viewpoint 
provides a clearer understanding of reality than a close one.

\begin{figure}[hb]
\centering
\hfill
   \includegraphics[angle=0,width=0.6\textwidth]{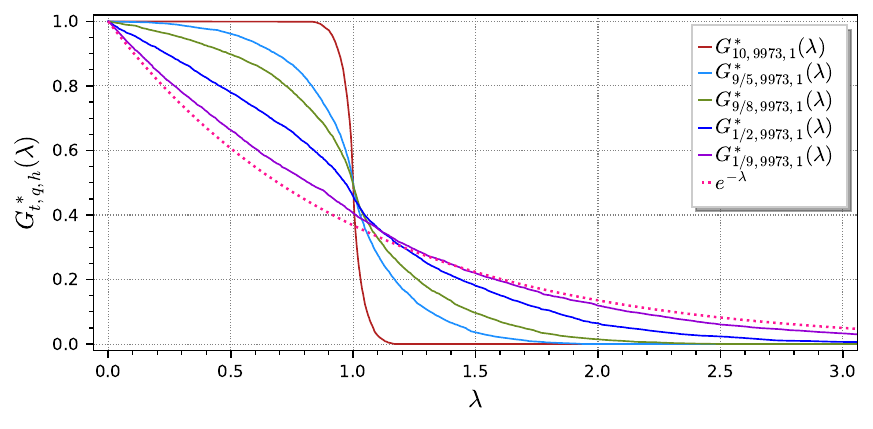}  
\hfill\mbox{}
\caption{The approximations $G^*_{t,p,h}(\lambda)$ of
$G_{t,\cP}(\lambda)$ approaching $e^{-\lambda}$ as $t$ decreases
to $1/J$.
In the graphs, $p=9973$ is prime, $h=1$ and $t\in\{10,9/5,9/8,1/2,1/9\}$.}
 \label{FigureGtoRandom}
 \end{figure}

We conclude by noting that the function $G_{t,\cP}(\lambda)$ obtained in Theorem~\ref{Thm1.1} 
closely matches the gap distribution function for the analogue modular curves 
whose points are randomly distributed under the condition that
just one point is placed above each integer abscissa.
Otherwise, for other curves that have multiple points above each integer abscissa,
the gap distribution functions are
essentially distinct.

\end{document}